\documentclass[a4paper]{amsart}

\usepackage{graphicx}
\usepackage{amssymb,pstricks}
\usepackage[totalwidth=17cm,totalheight=24cm]{geometry}

\begin{document}

\newtheorem{cor}{Corollary}[section]
\newtheorem{thm}[cor]{Theorem}
\newtheorem{prop}[cor]{Proposition}
\newtheorem{lemma}[cor]{Lemma}
\theoremstyle{definition}
\newtheorem{defi}[cor]{Definition}
\theoremstyle{remark}
\newtheorem{remark}[cor]{Remark}
\newtheorem{example}[cor]{Example}
\newtheorem{question}[cor]{Question}

\newcommand{\cD}{{\mathcal D}}
\newcommand{\cM}{{\mathcal M}}
\newcommand{\cT}{{\mathcal T}}
\newcommand{\cML}{{\mathcal M\mathcal L}}
\newcommand{\cGH}{{\mathcal G\mathcal H}}
\newcommand{\C}{{\mathbb C}}
\newcommand{\N}{{\mathbb N}}
\newcommand{\R}{{\mathbb R}}
\newcommand{\Z}{{\mathbb Z}}
\newcommand{\Kt}{\tilde{K}}
\newcommand{\Mt}{\tilde{M}}
\newcommand{\St}{\tilde{S}}
\newcommand{\dr}{{\partial}}
\newcommand{\kappab}{\overline{\kappa}}
\newcommand{\pib}{\overline{\pi}}
\newcommand{\Sigmab}{\overline{\Sigma}}
\newcommand{\gd}{\dot{g}}
\newcommand{\diff}{\mbox{Diff}}
\newcommand{\dev}{\mbox{dev}}
\newcommand{\devb}{\overline{\mbox{dev}}}
\newcommand{\devt}{\tilde{\mbox{dev}}}
\newcommand{\vol}{\mbox{Vol}}
\newcommand{\hess}{\mbox{Hess}}
\newcommand{\db}{\overline{\partial}}
\newcommand{\Sigmat}{\tilde{\Sigma}}
\newcommand{\lambdat}{\tilde{\lambda}}

\newcommand{\cunc}{{\mathcal C}^\infty_c}
\newcommand{\cun}{{\mathcal C}^\infty}
\newcommand{\dd}{d_D}
\newcommand{\dmin}{d_{\mathrm{min}}}
\newcommand{\dmax}{d_{\mathrm{max}}}
\newcommand{\Dom}{\mathrm{Dom}}
\newcommand{\dn}{d_\nabla}
\newcommand{\ded}{\delta_D}
\newcommand{\delmin}{\delta_{\mathrm{min}}}
\newcommand{\delmax}{\delta_{\mathrm{max}}}
\newcommand{\hmin}{H_{\mathrm{min}}}
\newcommand{\maxi}{\mathrm{max}}
\newcommand{\oL}{\overline{L}}
\newcommand{\oP}{{\overline{P}}}
\newcommand{\Ran}{\mathrm{Ran}}
\newcommand{\tgamma}{\tilde{\gamma}}
\newcommand{\cotan}{\mbox{cotan}}

\newcommand{\II}{I\hspace{-0.1cm}I}
\newcommand{\III}{I\hspace{-0.1cm}I\hspace{-0.1cm}I}

\title{On weakly convex star-shaped polyhedra}
\author{Jean-Marc Schlenker}
\address{Institut de Mathématiques de Toulouse, UMR CNRS 5219,
Universit\'e Toulouse III,
31062 Toulouse cedex 9, France }
\email{schlenker@math.ups-tlse.fr}
\date{April 2007 (v1)}

\begin{abstract}
Weakly convex polyhedra which are star-shaped with respect to 
one of their vertices are infinitesimally rigid. This is partial answer
to the question whether every decomposable weakly convex polyhedron is
infinitesimally rigid. The proof uses
a recent result of Izmestiev on the geometry of convex caps.
\end{abstract}

\maketitle


\section{Introduction}

\subsection{The rigidity of convex polyhedra.}

The rigidity of Euclidean polyhedra has been of interest to
geometers since 
Legendre \cite{legendre} and Cauchy \cite{cauchy} proved that
convex polyhedra are globally rigid. This result was an important
source of inspiration in subsequent geometry, for instance 
for the theory of convex
surfaces, and was a key tool in Alexandrov's theory of isometric
embeddings of polyhedra \cite{alex,Po}. 

The notion of global rigidity leads directly to the related notion
of infinitesimal rigidity; a polyhedron is {\it infinitesimally 
rigid} if any non-trivial first-order deformation induces a non-zero
variation of the metric on one of its faces. Infinitesimal rigidity
is important in applications since a structure which is rigid but
not infinitesimally rigid is likely to be physically unreliable.
Although Cauchy's 
argument can be used to prove that convex polyhedra are infinitesimally
rigid, this result was proved much later by M. Dehn \cite{dehn-konvexer},
by completely different methods.

\subsection{Non-convex polyhedra.}

Cauchy's theorem left open the question of rigidity of non-convex
polyhedra, until examples of flexible polyhedra were constructed
by Connelly \cite{connelly}. It would however be interesting to
know a class of rigid polyhedra wider than the convex ones.
We say that a polyhedron is {\it weakly convex} if its vertices
are the vertices of a convex polyhedron, and that it is 
{\it decomposable} if it can be cut into convex polyhedra without
adding any vertex.

\begin{question}\label{q:main}
Let $P$ be a weakly convex, decomposable polyhedron. Is $P$ 
infinitesimally rigid ? 
\end{question}

This question came up naturally in \cite{rcnp}, where it was proved,
using hyperbolic geometry tools, that the result is positive if 
the vertices of $P$ are on an ellipsoid, or more generally if 
there exists an ellipsoid which contains no vertex of $P$ but
intersects all its edges. It was proved in \cite{vienna} that
the answer is also positive for two other classes of 
polyhedra: suspensions -- which can be cut into simplices
with only one interior edge -- and polyhedra which have at
most one non-convex edge, or two non-convex edges sharing a 
vertex.

\subsection{Main result.}

Here we extend the result of \cite{vienna} to a wider class
of weakly convex decomposable polyhedra. From here on all the
polyhedra we consider are triangulated; it is always
possible to reduce to that situation by adding ``flat'' 
edges in non-triangular faces. 
For a polyhedron $P$ which is star-shaped with respect
to a vertex $v_0$, we do this subdivision by decomposing 
all non-triangular faces adjacent $v_0$ by adding only
diagonals containing $v_0$, so that this refinement of the
triangulation of the boundary of $P$ is compatible with
a triangulation of the interior of $P$ for which all simplices
contain $v_0$.

\begin{thm}\label{tm:main}
Let $P$ be a weakly convex polyhedron, which is star-shaped
with respect to one of its vertices. Then $P$ is infinitesimally
rigid.
\end{thm}

Here by ``star-shaped'' with respect to $v_0$ we mean that the
interior of $P$ has a decomposition as the union of finitely many
non-degenerate simplices, all containing $v_0$ as one of their
vertices, of disjoint interior, and such that the intersection 
of each with $P$ is a face of both.

\subsection{A refined statement.}

There is a slightly refined version of Theorem \ref{tm:main}, 
giving a better understanding on the reasons for which rigidity holds.
Let $P$ be star-shaped with respect to a vertex $v_0$, and let
$P=S_1\cup\cdots\cup S_N$ be a triangulation of $P$ as a union
of non-degenerate
simplices all containing $v_0$ (and having disjoint interior).
Let $e_1, \cdots, e_m$ be the interior edges of this triangulation,
i.e., the edges of the $S_j$ which are not contained in faces 
of $P$. Let $l_i$ be the length of $e_i$.

It is then possible to consider a wider class of (small) deformations
of the metric on the interior of $P$: those for which the $l_i$ vary, while 
the length of the edges of $P$ remain constant. Such a 
variation of the $l_i$ determines a unique deformation of
the metric on the $S_j$, which can however still be glued
isometrically along their common faces. Under such a 
variation, cone singularities might appear along the $e_i$:
the angles around those edges might become different from $2\pi$.
We call $\theta_i$ the angle around $e_i$. 

\begin{defi}\label{df:lambda}
Let
$$ \Lambda_P := 
\left(\frac{\dr \theta_i}{\dr l_j}\right)_{1\leq i,j\leq m}~. $$
\end{defi}

Note that $\Lambda_P$ {\it a priori} depends also on the decomposition
$P=S_1\cup\cdots\cup S_N$ (and of the labeling of the $e_j$). It
is well-known that $\Lambda_P$ is symmetric
(see e.g. \cite{vienna}), this follows from the fact that 
$\Lambda_P$ is minus the Hessian of the total
scalar curvature of the metrics obtained by varying the $l_i$.
The following statement is also well known.

\begin{remark} \label{rk:rigid}
$P$ is infinitesimally rigid if and only if $\Lambda_P$ is
non-degenerate.
\end{remark}

The proof is elementary: isometric first-order deformations of $P$
correspond precisely to first-order variations of the $l_i$
which do not change, at first order, the $\theta_i$. Although the
proof of this point requires some care, we do not include one here
and refer the reader to \cite{izmestiev,bobenko-izmestiev} where
a similar problem is treated in full details.

\begin{thm} \label{tm:lambda}
Under the hypothesis of Theorem \ref{tm:main}, $\Lambda_P$
is positive definite.
\end{thm}

\subsection{A word on the proof.}

The proof is only indirectly related to the argument used in 
\cite{rcnp}, and different from those used in \cite{vienna}.
It is based on a recent result of Izmestiev \cite{izmestiev}, who
gives a new proof of Alexandrov's theorem on the existence 
and uniqueness of 
a polyhedral convex cap with a given induced metric, based on
the concavity of a geometric function. We slightly extend his
argument, to encompass weakly convex ``caps'', by proving that
``removing'' a simplex to a (weakly) convex cap actually makes
this function ``more'' concave -- a point which we found somewhat
surprising.
We then use a
classical projective argument to obtain Theorem \ref{tm:main}.
Theorem \ref{tm:lambda} follows from the same arguments.

\section{Weakly convex hats}

We use a notion of ``convex cap'' which is a little different
from the one used by Izmestiev \cite{izmestiev}. We will use a 
different name to avoid ambiguities. In the whole paper we
consider a distinguished oriented plane, which can for
convenience be taken to be the horizontal plane $\{ z=0\}$. 
We call $\R^3_+$ the half-space bounded by this plane $\{ z=0\}$
on the the side of its oriented normal.

\begin{defi}
Let $E\subset \R_+^3$, its {\it shadow} $Sh(E)$ is the set of
points $m\in \R^3_+$ for which there exists a point $m'\in E$
such that $m$ is contained in the segment joining $m'$ to
its orthogonal projection on the horizontal plane $\{ z=0\}$.
\end{defi}

\begin{defi}
A {\bf convex hat} is a polyhedral surface $H$ in $\R_+^3$ such that 
\begin{enumerate}
\item $H$ is homeomorphic to a disk and has finitely many vertices,
\item no point of $H$ is in the shadow of another,
\item $H$ is contained in the boundary of the convex hull of 
$Sh(H)$.
\end{enumerate}
\end{defi}

Note that we do {\it not} demand that $Sh(H)$ is convex, and that the
projection of $H$ on $\{ z=0\}$ is not necessarily convex.

\begin{defi}
A {\bf weakly convex hat}  is a polyhedral surface $H\subset \R_+^3$,
satisfying conditions (1) and (2) of the previous definition, and
such that every vertex of $H$ is an extremal point of the convex
hull of $Sh(H)$.
\end{defi}

A convex hat can be obtained by the following procedure. Start
from a convex polyhedron $P$, and apply a projective transformation
sending one of the vertices, $v$, to infinity in the vertical
direction (towards $z\rightarrow -\infty$). Then remove all
edges and faces adjacent to $v$. Any convex hat which has as 
its projection
on $\{ z=0\}$ a convex polygon can be obtained
in this manner. In the same way, one can start from a weakly
convex polyhedron which is star-shaped with respect to one of
its vertices, say $v_0$, and apply a projective map sending
$v_0$ to infinity, to obtain a weakly convex hat. Any 
weakly convex hat can be obtained in this manner.

The proof of Theorem \ref{tm:main} will follow from the following
lemma, using classical arguments relating projective transformations
to infinitesimal rigidity.

\begin{lemma}\label{lm:main}
Let $H$ be a weakly convex hat. Any first-order isometric
deformation of $H$ which fixes the heights of the boundary
vertices is trivial.
\end{lemma}

The proof relies on an extension of the notion of (weakly) convex 
cap, to allow for cone singularities along vertical edges, a 
device which is common in the field known as ``Regge calculus'',
or more specifically e.g. in \cite{bobenko-izmestiev},
\cite{izmestiev}, or in the last part of \cite{vienna}.
This type of construction has been used successfully also in the
contexts of hyperbolic polyhedra or circle patterns on surfaces, 
see e.g. \cite{Ri2,rivin-annals,leibon1,bobenko-springborn,rcnp,cpss}.

\begin{defi}
A {\it prism} is a non-degenerate convex polyhedron $P$ in 
$\R^3_+$ which is the shadow of a triangle in $\R^3_+$.
\end{defi}

The faces which are neither the bottom or the upper face of 
$P$ are its {\it vertical} faces. It is not difficult to check
that a prism is uniquely determined, among prisms with the same
induced metric on the upper face, by the heights of its vertical
edges.

\begin{defi}
A {\it generalized hat} is a metric space obtained 
from a finite set of prisms $P_1, \cdots, P_N$ by
isometrically identifying some of their vertical faces, so 
that
\begin{itemize}
\item each vertical face is glued to at most one other, so
that singularities occur only at line segments corresponding
to some vertical edges of the $P_i$,
\item the prisms containing a given vertical edge are pairwise
glued along vertical faces in a cyclic way (with either all 
vertical faces containing the given vertical edges pairwise 
glued, for an interior edge, or with two faces not glued and
corresponding to vertical boundary faces of the generalized hat,
for a boundary edge),
\item under the gluing of two vertical faces, 
the segments corresponding to the
bottom (resp. upper) face of the $P_i$ are identified.
\end{itemize}
\end{defi}

Given a convex or weakly convex hat $H$, it can be used to
construct a generalized hat $G$ by gluing the shadows of 
the faces of $H$. Moreover it's easy to
characterize the generalized hats obtained in this manner. It
is necessary that the angles around all interior ``vertical''
edges are equal to $2\pi$; under this condition, generalized
hats admit an isometric immersion into $\R^3_+$, with their
bottom faces sent to $\{ z=0\}$, and a generalized hat $G$ is
obtained from a (weakly) convex hat $H$ if and only if this
image in $\R^3$ is embedded and (weakly) convex.
This simple construction
allows us to consider convex or weakly convex hats as special
cases of generalized hats.

We define a generalized hat to be {\it convex} if it is convex at each
edge $e$ which is shared by the upper faces of two of the prisms
$P_i$ and $P_j$,
i.e., if the angles at $e$ of $P_i$ and $P_j$ add up
to at most $\pi$. It is {\it strictly convex} if those angles
add up to strictly less than $\pi$.

Given a generalized hat $G$, one can consider the space $\cM_G$ 
of all generalized hats for which the upper boundary has the
same combinatorics and the same induced metric. It is not 
difficult to check that $\cM_G$ is parametrized by the 
heights of the vertical edges $h_1, \cdots, h_n$. 
We call $\cM_{G,0}$ the subspace of $\cM_G$ of 
generalized hat having the same boundary heights as $G$,
so that $\cM_{G,0}$ is parametrized by the heights of the
interior vertical edges, $h_1, \cdots, h_m$.

\section{The rigidity of convex hats}

We mainly recall in this section results of Izmestiev \cite{izmestiev},
adapting the arguments to the proof of Lemma \ref{lm:main} for the
special case of convex hats. In the next section it is shown how 
the argument can be extended to weakly convex hats.

The proof is based on a matrix very similar to the matrix $\Lambda_P$
appearing in Definition \ref{df:lambda}. We consider a convex hat $H$,
and the corresponding generalized hat $G$. Since a prism, with given 
induced metric on its upper face, is uniquely determined by the
heights of its vertices, elements of $\cM_{G,0}$ are 
uniquely determined by the heights of the interior vertices. Conversely,
each choice of those heights, close to the heights of the interior
vertices in $H$, determines an element of $\cM_{G,0}$. We call 
$e_1, \cdots, e_m$ the vertical edges ending at interior points
of $H$, $(h_i)_{1\leq i\leq m}$ their heights, and $(\theta_i)_{1\leq
i\leq m}$ the angle around them. So the $\theta_i$ are equal to 
$2\pi$ for $G$, but not necessarily at other points of $\cM_{G,0}$.

\begin{lemma}[Izmestiev \cite{izmestiev}] \label{lm:lambda-hat}
Let 
$$ \Lambda_G := 
\left(\frac{\dr \theta_i}{\dr h_j}\right)_{1\leq i,j\leq m}~. $$
Then $\Lambda_G$ is symmetric and positive definite.
\end{lemma}

We only give a brief outline of the proof here. The symmetry of
$\Lambda_G$ follows from the fact that it is minus the Hessian of a
natural ``total scalar curvature'' function appearing in this
context, called $S$ in \cite{izmestiev}. The coefficients of 
$\Lambda_G$ are computed explicitly in \cite{izmestiev} (Proposition
4, note that the coefficients given there are minus the ones
considered here), they are equal to:
\begin{itemize}
\item $a_{ij}=0$ when $i\neq j$ and $e_i$ and $e_j$ are not the
endpoints of an interior edge of $H$.
\item $a_{ij}=-(\cotan(\alpha_{ij}) +
\cotan(\alpha_{ji)})/l_{ij}\sin^2(\rho_{ij})$ when $i\neq j$ but $e_i$
and $e_j$ are the two endpoints of an interior edge of $H$. Here $\alpha_{ij}$
and $\alpha_{ji}$ are the angles between the shadow of the edge 
of $G$ joining the endpoints of $e_i$ and $e_j$ with the two 
upper faces of $G$ adjacent to that edge, $l_{ij}$ is the length
of that edge, and $\rho_{ij}$ is its angle with the vertical.
\item $a_{ii} = -\sum_{j\neq i} a_{ij}$.
\end{itemize}
It follows from this explicit description that $\Lambda_G$ has
dominant diagonal, and therefore that it is positive definite.

Remark \ref{rk:rigid} still applies in this context, so that it
follows from Lemma \ref{lm:lambda-hat} that convex hats are
infinitesimally rigid.

\section{Weakly convex hats are rigid}

The proof of Lemma \ref{lm:main} follows from Lemma \ref{lm:lambda-hat}
by a simple argument, remarking that (1) it is possible to go from a
weakly convex hat to a convex hat by adding a finite set of simplices
(which are in a specific position with respect to the vertical direction)
(2) when removing such a simplex, the matrix $\Lambda$ defined 
above becomes ``more'' positive.

\begin{lemma} \label{lm:codecomp}
Let $H$ be a weakly convex hat, and let $H_c$ be the convex hat
which is the union of the upper faces of the convex hull of $Sh(H)$
which project orthogonally to $\{ z=0\}$ as a polygon in the 
projection of $H$.
There exists a finite sequence $H_0, \cdots, H_p$ of weakly convex
hats in $\R^3_+$ such that
\begin{itemize}
\item $H_0=H$ and $H_p=H_c$,
\item for all $i\in \{ 1, \cdots, p\}$, $H_i$ has the same vertices
as $H_{i-1}$, and $Sh(H_i)$ is obtained from $Sh(H_{i-1})$ by 
gluing a simplex $S_i$,
\item the projection of $S_i$ on $\{ z=0\}$ is a quadrilateral.
\end{itemize}
\end{lemma}

\begin{proof}
Set $H_0:=H$, and choose a concave edge $e_0$ of $H_0$  (which is
thus not a boundary edge of $H_0$) with vertices
$v_0$ and $v_1$. Let 
$f$ and $f'$ be the faces of $H$ adjacent to $e_0$, and let $v_3$ and
$v_4$ be the vertices of $f$ and $f'$ opposite to $e_0$. Let $S_1$
be the simplex with vertices $v_0, v_1, v_2, v_3$, then $S_1$ 
projects to $\{ z=0\}$ as a quadrilateral. 

We can add to $Sh(H)$ the simplex $S_1$, this yields a polyhedron 
in $\R^3_+$ which is the shadow of a weakly convex hat $H_1$ (which
by construction has the same vertices as $H_0$). 

If $H_1$ is convex, the lemma is proved. Otherwise, $H_1$ has at
least one concave edge and one can choose one of those edges, say 
$e_1$, and repeat the construction, adding a simplex $S_2$. 

After a finite number of steps the weakly convex hat $H_p$ 
obtained in this way will be convex, because the number of 
simplices that can be added is bounded from above, for instance
by the number of Euclidean simplices having as vertices some
vertices of $H$.
\end{proof}

The next step is to describe in what manner the matrix $\Lambda$
associated to a weakly convex hat changes when a simplex is
removed. We consider a simplex $S$ with vertices $v_1, v_2,
v_3, v_4$ which projects on the plane $\{ z=0\}$ as a quadrilateral.
Then the boundary of $S$ is the union of two surfaces, each made
by gluing two triangles, and each of which has injective projection
on $\{ z=0\}$: the ``lower'' surface $S_-$, and the ``upper'' surface
$S_+$, with $S_-\subset Sh(S_+)$. 
We suppose for instance that $S_-$ is the union of the
triangles $(v_1,v_3,v_4)$ and $(v_2, v_3,v_4)$, while $S_+$
is the union of $(v_1,v_2,v_3)$ and $(v_1,v_2,v_4)$. Let $h_i$
be the height of $v_i$ over $\{ z=0\}$. Any first order variation
of the $h_i$, $1\leq i\leq 4$, determines a first-order displacement
of the $v_i$ which preserves the lengths of the five segments in 
$S_-$ (including the diagonal), which is unique up to horizontal
translation and rotation with vertical axis. Similarly a first-order 
variation of the $h_i$ determines a displacement of the $v_i$ which
preserves the lengths of the five segments in $S_+$. 

\begin{defi}
Let 
$$ M_S = \left(\frac{\dr (\theta_i^- - 
\theta_i^+)}{\dr h_j}\right)_{1\leq i,j\leq 4}~, $$
where, for heights of the $v_i$ close to the $h_i$, $\theta_i^+$ is
the angle of the projection of $S_+$ on $\{ z=0\}$ at the projection of
$v_i$, and $\theta^-_i$ is the angle of the projection of $S_-$ at the
projection of $v_i$.
\end{defi}

\begin{lemma} \label{lm:matrice}
Let $H$ and $H'$ be two weakly convex hats, with $Sh(H')$ obtained 
by removing from $Sh(H)$ a simplex $S$.
Then $\Lambda_{H'}$ is obtained by adding $M_S$ to $\Lambda_H$ (with
the lines/columns of $M_S$ added to the lines/columns of $\Lambda_H$
corresponding to the same vertices).
\end{lemma}

\begin{proof}
This follows from the definitions, since $\Lambda_{H'}$ is equal to
$\Lambda_H$ except that the variation of the curvature at the 
vertical edges ending on the vertices of $S$ are given by the 
lower surface $S_-$ rather than by the upper surface $S_+$.
\end{proof}

The interesting point is that $M_S$ is always positive semi-definite,
so that adding it to a positive definite matrix yields another positive
definite matrix.

\begin{lemma} \label{lm:simplexe}
For any simplex $S$ projecting on $\{ z=0\}$ as a quadrilateral,
$M_S$ is positive semi-definite of rank $1$.
\end{lemma}

\begin{proof}
The space of Killing fields in $\R^3$ has dimension $6$.
It contains a 3-dimensional subspace fixing $\{ z=0\}$,
and therefore acting on $S$ without changing any of the
heights. There remains a 3-dimensional vector space of
Killing fields which do change the heights $h_i$. Each
acts by deforming globally $S$, so that, at each vertex,
the angles of the projection of the upper and the lower
surface change in the same way, and therefore those
first-order variations of the $h_i$ are in the kernel
of $M_S$. So the rank of $M_S$ is at most $1$.

The same argument can be used, conversely, to show that
the rank of $M_S$ can not be zero. Otherwise the kernel of
$M_S$ would have dimension $4$, which would mean that there
exists a non-trivial first-order deformation of $S$ 
leaving invariant the lengths of all edges in both the upper
and lower surfaces, and such that the angles of the 
projections of the upper and lower surface vary in the 
same way. One could then consider the first-order deformations
of the upper and of the lower surface, and add a trivial 
deformation so that they match at all four vertices, because
the first-order variations of both the heigths of the vertices
and the projections of the two surfaces on $\{ z=0\}$ match. 
This would mean that there is a non-trivial
isometric deformation of $S$, and this is well known to
be impossible -- all simplices are infinitesimally rigid.

So the signature of $M_S$ is constant over the space of
simplices which projects on $\{ z=0\}$ as quadrilaterals.
This means that $M_S$ is either positive semi-definite or
negative semi-definite for all such simplices. To decide
which happens, it is sufficient to check for one simplex,
for instance a maximally symmetric one. Consider the 
first-order deformations pictured in Figure 1, with
the heights of $v_1$ and $v_2$ raised and the heights of
$v_3$ and $v_4$ lowered. 

\begin{figure}[htbp]
   \begin{center}
      \includegraphics[width=9cm]{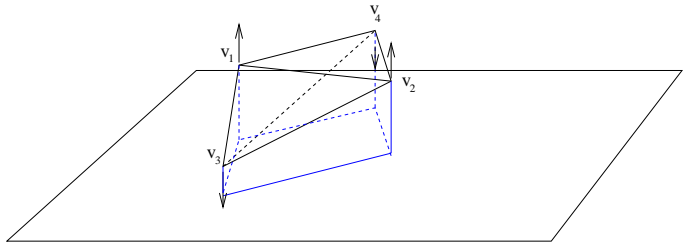}
   \end{center}
   \caption{A positive deformation of a simplex}
\end{figure}

It is easy to check that in this
case:
\begin{itemize}
\item $\theta_1^+$ and $\theta_2^+$ decrease: the angles
of the projection of the upper surface at the projections of
$v_1$ and $v_2$ decrease,
\item $\theta_3^+$ and $\theta_4^+$ increase,
\item $\theta_1^-$ and $\theta_2^-$ increase,
\item $\theta_3^-$ and $\theta_4^-$ decrease,
\end{itemize}
It follows that the first-order variation of $\theta^-_i-\theta^+_i$ 
is positive at $v_1$ and $v_2$ and negative at $v_3$ and $v_4$, so
that $M_S$ has at least one positive eigenvalue. So $M_S$ is 
positive semi-definite.
\end{proof}

\begin{lemma}\label{lm:cvx-weak}
Let $H$ be any weakly convex hat, then $\Lambda_H$ is positive definite.
\end{lemma}

\begin{proof}
Let $H_c$ be the convex hat obtained as the upper boundary of $Sh(H)$.
Lemma \ref{lm:codecomp} shows that $Sh(H)$ is obtained from $Sh(H_c)$
by removing a finite sequence of simplices. But $\Lambda_{H_c}$
is positive definite by Lemma \ref{lm:lambda-hat}, and Lemma 
\ref{lm:matrice} shows that, each time a simplex is removed,
the matrix $\Lambda_H$ changes by the addition of a $4 \times 4$
matrix, which is positive semi-definite by Lemma \ref{lm:simplexe}. 
It follows that $\Lambda_H$ is also positive definite.
\end{proof}

\begin{proof}[Proof of Lemma \ref{lm:main}]
We have already seen that the fact that $\Lambda_H$ is non-degenerate
implies that $H$ is infinitesimally rigid: any isometric
first-order deformation of $H$ which fixes the boundary heights is
trivial.
\end{proof}

\section{Projective maps}

The goal of this section is to prove Theorem \ref{tm:main}, concerning
polyhedra which are star-shaped with respect to one of their vertices,
using Lemma \ref{lm:main}, which deals with weakly convex hats.
The basic idea here is old, going back at least to Darboux 
\cite{darboux12} and Sauer \cite{sauer}: infinitesimal rigidity
is a property which is invariant under projective maps. The particular
case of this property which is used here can be stated more precisely
as follows.

\begin{lemma} \label{lm:darboux}
Let $v_0\in \R^3\subset \R P^3$, and let $\phi:\R P^3\rightarrow \R P^3$
be a projective transformation sending $v_0$ to the point at infinity
corresponding to the vertical direction in $\R^3$. There exists a 
map $\Phi:T\R^3\rightarrow T\R^3$ sending $(x,v)\in T\R^3$ to 
$(\phi(x),\psi_x(v))\in T\R^3$ such that:
\begin{itemize}
\item the image by $\Phi$ of any Killing vector field in $\R^3$ is
a Killing vector field,
\item Killing fields which are infinitesimal rotations of axis
containing $v_0$
are sent to the translations along horizontal directions and the
infinitesimal rotations of vertical axis.
\end{itemize}
\end{lemma}

The proof of this Lemma is left to the reader, since it is quite
classical. The map $\psi_x$ can be explicitly described as follows:
it sends vectors parallel to the direction of $v_0$ to vertical vectors
of the same norm, while acting on vectors orthogonal to the direction
of $v_0$ as the differential of the projective
map $\phi$.

\begin{proof}[Proof of Theorem \ref{tm:main}]
Let $P$ be a weakly convex polyhedron which is star-shaped with 
respect to a vertex $v_0$. Let $U$ be an isometric
first-order deformation of $P$, i.e., the restriction of $V$
to each face of $P$ is a Killing field. Adding a global Killing
field if necessary, we can assume that the restriction of $U$ to
all faces of $P$ containing $v_0$ is a Killing field fixing $v_0$,
i.e., an infinitesimal rotation with axis containing $v_0$.

Let $Q=\phi(P)$, then $Q$ is an infinite polyhedron with 
infinite vertical faces corresponding to the faces of $P$
containing $v_0$. Applying a vertical translation if necessary,
we can suppose that the intersection of $Q$ with $\R^3_+$ is 
of the form $Sh(H)$, where $H$ is a weakly convex hat with 
one face in its upper boundary corresponding to each face of 
$P$ not containing $v_0$ (and conversely).

Now let $V=\Phi(U)$, then, by Lemma \ref{lm:darboux}, the restriction
of $V$ to each face of $Q$ is a Killing field, so that $V$ is a
first-order isometric deformation of $Q$. Moreover, since the
restriction of $U$ to each face of $P$ containing $v_0$ fixes
$v_0$, the restriction of $V$ to the vertical faces of $Q$
are horizontal translations or rotations around a vertical axis. 
So $V$ does not change the heights of
the boundary vertices of $H$. It follows from Lemma \ref{lm:main}
that $V$ is a trivial deformation -- the restriction to $Q$
of a global Killing vector field -- and therefore, again from
Lemma \ref{lm:darboux}, that $U$ is a trivial deformation of $P$.
So $P$ is infinitesimally rigid.
\end{proof}

Note that this argument -- along with the results recalled in section 3, 
but without the need of section 4 -- gives a direct proof of the 
infinitesimal rigidity of convex polyhedra.

\begin{proof}[Proof of Theorem \ref{tm:lambda}]
Let again $Q=\phi(P)$, so that, applying a vertical translation again
if necessary, $Q\cap \R^3_+=Sh(H)$, where $H$ is a weakly convex hat.
Let $(\phi_t)_{t\in [0,1]}$ be a one-parameter family of projective
transformation, chosen such that
\begin{itemize}
\item $\phi_0$ is the identity, while $\phi_1=\phi$,
\item $\phi_t(P)$ is a compact polyhedron in $\R^3$ for all $t\in [0,1)$.
\end{itemize}
Let $P_t=\phi_t(P), 0\leq t<1$. We know by Theorem \ref{tm:main}
that $P_t$ is infinitesimally rigid for all $t\in
[0,1)$. This means by Remark \ref{rk:rigid} that $\Lambda_{P_t}$ has
maximal rank, so that the signature of $\Lambda_{P_t}$ is constant
for $t\in [0,1)$.

But a quick look at the definitions shows that $\lim_{t\rightarrow 1}
\Lambda_{P_t}=\Lambda_H$, which is positive definite by Lemma 
\ref{lm:main}. It follows that $\Lambda_P=\Lambda_{P_0}$ is also
positive definite.
\end{proof}

\subsection*{Acknowledgements}

I'm grateful to Bob Connelly, Fran\c{c}ois Fillastre and Ivan Izmestiev
for useful conversations, and for relevant remarks on the first draft
of this text.

\bibliographystyle{amsplain}

\def\cprime{$'$}
\providecommand{\bysame}{\leavevmode\hbox to3em{\hrulefill}\thinspace}
\providecommand{\MR}{\relax\ifhmode\unskip\space\fi MR }
\providecommand{\MRhref}[2]{%
  \href{http://www.ams.org/mathscinet-getitem?mr=#1}{#2}
}
\providecommand{\href}[2]{#2}

\end{document}